\documentclass[a4paper,12pt,oneside]{article}
\usepackage[english]{babel}
\usepackage[T1]{fontenc} 
\usepackage[utf8]{inputenc}
\usepackage{amsthm}
\usepackage{amsmath}
\usepackage{amssymb}  
\usepackage{indentfirst}
\usepackage{fancyhdr}
\usepackage{amsthm}
\usepackage{graphicx}
\usepackage{pdfpages}
\usepackage{esint}
\usepackage{url}

\numberwithin{equation}{section}

\theoremstyle{plain} 
\newtheorem{thm}{Theorem}[section]

\newtheorem{prop}[thm]{Proposition} 
\newtheorem{rmk}[thm]{Remark}

\usepackage{geometry}
\allowdisplaybreaks[4]

\begin{document}

\title{\textbf{Regularity for minimizers of scalar integral functionals}}

\author {   
Antonio Giuseppe Grimaldi \thanks{Dipartimento di Scienze Matematiche "G. L. Lagrange",
Politecnico di Torino, Corso Duca degli Abruzzi, 24, 10129 Torino, 
 Italy. E-mail: \textit{antonio.grimaldi@polito.it}} 
  , Elvira Mascolo \thanks{Universit\`{a} degli Studi di Firenze,
viale  Morgagni, 41125 Firenze, Italy.
E-mail: \textit{elvira.mascolo@unifi.it}}
 , Antonia Passarelli di Napoli
\thanks{Dipartimento di Matematica e Applicazioni "R.
Caccioppoli", Universit\`{a} degli Studi di Napoli ``Federico
II", via Cintia - 80126 Napoli.
E-mail: \textit{antpassa@unina.it}}
}

\maketitle
\maketitle

\begin{abstract}
We prove the local Lipschitz regularity of the local minimizers of scalar integral functionals of the form \begin{equation*}
\mathcal{F}(v;\Omega)= \int_{\Omega} f (x, Dv) dx
\end{equation*}
under $(p,q)$-growth conditions. The main novelty is that, beside a suitable regularity assumption on the partial map $x\mapsto f(x,\xi)$, we do not assume any special structure for the energy density as a function of the $\xi$-variable.
\end{abstract}



\section{Introduction}
\label{Introduction}
The classical problem of the Calculus of Variations is formulated as finding a function $u$ assuming a given value $u_0$  at the boundary $\partial \Omega$ of an
open regular bounded set $\Omega \subset \mathbb{R}^n$ which minimizes an integral functional of the form
\begin{equation}\label{functional}
\mathcal{F}(v;\Omega)= \int_{\Omega} f (x, Dv) dx
\end{equation}
among all functions $ v :\Omega \subset \mathbb{R}^n\rightarrow \mathbb{R}$ , assuming the  boundary value  $u_0$.
 The 
functional space where to get the solutions depends on the growth conditions of $f = f (x, \xi)$
as  $|\xi| \rightarrow +\infty$. 
 It is well known that, if the energy density satisfies the following lower bound 
$$
f (x, \xi) \geq C_1|\xi|^p ,
$$
for some exponent   $p>1$, the right functional space is the Sobolev class $W^{1,p}(\Omega)$. If the partial map  
$\xi\mapsto f (x, \xi)$ is convex (strictly convex), the semicontinuity property gives a sufficient conditions  to obtain the existence (and uniqueness) of minimizers in $W_{0}^{1,p}(\Omega)+u_0$. 

A different question is the regularity of the minimizer: the problem is to find  the additional regularity properties to impose on $f = f (x, \xi)$ in \eqref{functional}, in order to obtain regular minimizer at least  in the interior of $\Omega$.

Let us consider integrand $f = f (x, \xi)$ satisfying $(p,q)$-growth conditions  
of the type 
\begin{equation}\label{p,q-growth}
  C_1 |\xi|^p\le f (x, \xi) \leq C_2(1+|\xi|)^q, \quad  \,\,  \textrm{a.e.} \,\, x \in \Omega,\,\, \xi \in \mathbb{R}^n  
\end{equation}
for $ 1 <  p\le q$ and $C_1,C_2$  positive constants. 
In the case $p=q$, i.e. 
$$ 
C_1|\xi|^{p}\leq f(x,\xi)\leq C_2(1+|\xi|)^{p}, 
$$
the regularity of the minimizers is nowadays a classical topic and we refer the interested reader, for example, to the monograph \cite{Giusti} and the references therein.


However, there are many integral functionals (and the related Euler-Lagrange equations) which satisfy $(p,q)$-growth as in \eqref{p,q-growth}
 with $1<p<q$ and their 
study  is motivated both from a mathematical point of view and from several models that arise from different problems in mathematical physic. We give some examples:

\begin{itemize}
\item Perturbation of polynomial growth: 
$$
f(\xi)=|\xi|^{p}\log^{\alpha}(1+|\xi|), \quad \, p \ge 1,\, \alpha>0
$$

\item Anisotropic growth 
$$
f(\xi) = \sum_{i=1}^n |\xi_{i}|^{p_{i}}, \quad \,p_i \ge p,\, \forall i=1,\dots,n
$$

\item Double phase functional: Model for strongly anisotropic 
materials (see \cite{zhikov1})
$$
f(x,\xi)= |\xi|^{p}+a(x)|\xi|^{q}, \quad \, 0 \le a(x)
$$

\item Variable exponent: Model proposed for electrorheological fluids by Rajagopal and Ru\v {z}i\v{c}ka \cite{raru}, {image restoration by Chen et al.\ \cite{chen}}, growth of heterogeneous sandpiles by Bocea et al.\ \cite{bocea}
$$
f(x,\xi)= |\xi|^{p(x)} \,\, \textrm{and} \,\, f(x,\xi)=[h(|\xi|)]^{p(x)}, \quad \, p \le p(x)\le q
$$
\item Anisotropic variable exponents
$$
f(x, \xi) \sim \sum_{i=1}^{n}|\xi_i|^{p_{i}(x)}, \quad \, \xi= (\xi_1,.....,\xi_n), \, p\leq p_i(x)\leq q
$$

\end{itemize}
Let us recall that, under $(p,q)$-growth conditions, a function $u\in W_{\rm loc}^{1,1}(\Omega)$ is a {\em  local minimizer} of
\eqref{functional} iff {$f(x,Du)\in L^1_{\mathrm{loc}}(\Omega)$ and} 
\begin{equation*}
\mathcal{F}(u;\text{supp\ }\varphi)\le  \mathcal{F}(u+\varphi;\text{supp\ }\varphi),
\end{equation*}
for all $\varphi\in W^{1,1}(\Omega)$ with $\operatorname{supp}\varphi\Subset \Omega$.  
\\
It is now well understood that, in order to obtain some regularity for the local minimizers of functionals with $(p,q)$-growth, $p<q$,  a restriction between the exponents $p$ and $q$ must be imposed (see the  counterexamples by Marcellini, \cite{mar91}.

Indeed, the gap $\frac{q}{p}$ cannot differ to much from $1$ and
sufficient conditions to the regularity can be expressed as 
$$\frac{q}{p}\le c(n)\displaystyle \to 1\,\,\textnormal{as}\,\,\,\,n\to \infty$$
{according to \cite{mar87, mar91}. }

The study of regularity properties of local minima to functionals with general growth started with the pioneering papers by Marcellini \cite{mar} , see also \cite{mar91, mar93}. In recent years there has been a considerable of interest in variational problems under nonstandard growth conditions.

A very interesting model case of functionals with $(p,q)$-growth is the so called double phase, whose energy density has the form 
\begin{equation*}
    f(x,\xi):=|\xi|^p+a(x)|\xi|^q
\end{equation*}
with $a(x) \in \mathcal{C}^{0,\alpha}(\Omega)$. This model case has been widely investigated in {\cite{BCM,colmin,colmin2, 25dfm,FMM}} with the aim of identifying sufficient and necessary conditions on the relation between $p$, $q$ and $\alpha$ to establish the H\"older continuity of the gradient of the local minimizers.

 The sharp condition on the gap is given by
$$\dfrac{q}{p} \le 1+ \dfrac{\alpha}{n}. $$
{The subject has been further developed in the recent articles by De Filippis and Mingione \cite{DM,Inventiones}.}
{For the study of the boundedness of minimizers we refer to  \cite{fuscsbor2, fussbo93} and more recently \cite{Cupini-Marcellini-Mascolo 2023}.  }

Other papers that deserve to be quoted are \cite{cupmarmas4,EMM1,el-pas1, el-pas2} for the case of elliptic equations and  \cite{boe1,boe2,boe3} for the case of parabolic equations and systems under $(p,q)$-growth.
{For a complete overview on the subject, we refer to the recent survey \cite{mar20-2, mar20-3, Mingione-Radulescu}.
}


It has also to be noted  that, in case of $(p,q)$-growth conditions, the presence of the $x$-variables cannot be treated as a simple perturbation. Actually,  the Lavrentiev phenomenon can occur and this is is a {{clear}} obstruction to the regularity of the minimizers. 
{In the very recent paper \cite{DefilLeo}, the authors proved the absence of the Lavrentiev phenomenon in case the energy density can be suitably approximate from below with a sequence of smooth integrands with standard growth conditions.}

Usually, the regularity of the minimizers follows  by means of \textit{a suitable  approximation procedure},
through the following steps:
\begin{itemize}
\item The first step consists in  proving an \textit{a priori} estimate for \textit{smooth minimizers} of $\mathcal{F}(u)$ with  constants   depending only on those appearing in the original assumptions;
\vskip.2cm
\item The second step consists in constructing a suitable  approximation  of the given integral  by means a sequence of regular $\mathcal{F}_h(u)$ with a unique \textit{smooth}  minimizers $u_h$;
\vskip.2cm
\item  The final  step consists in applying the a priori estimate to each  $u_h$,  and in showing that the sequence $(u_h)$  converges to a  minimizer of the original functional, which preserves the regularity properties. 
\end{itemize}

We {notice} that the approximation methods depend on the problem itself and it could be{ rather} different for different problems. Moreover it is worth to point out that the limiting process can be very delicate.

{While this procedure is well settled for functionals with the so called Uhlenbeck structure, i.e. $f(x,\xi)=\tilde f(x,|\xi|)$, for  energy densities  $f=f(x,\xi)$ without any additional structure conditions,} the  second two steps can be difficult {even in the scalar case}.

{The aim of this paper is to give a contribution to this problem, 
identifying   suitable assumptions on the partial map $x\mapsto f(x,\xi)$ and on the gap $\frac{q}{p}$}  which allow to approximate the  energy density $f(x,\xi)$ in \eqref{functional} with functionals of the form 
\begin{equation}
\label{specialcase}
f_h(x, \xi) = \sum_{i = 1}^N a_{i, h}(x) g_{i ,h}(\xi),  
\end{equation}
with $a_{i,h}(x) > 0$ a.e.\ in $\Omega$,  $g_{i,h}: \mathbb{R}^n \rightarrow [0, + \infty)$ strictly convex in $\xi$.


{It is worth pointing out that the approximating integrands usually satisfy standard growth conditions, but {in this case}  they  satisfy the same $(p,q)$-growth assumptions of the {original} integrand in \eqref{functional}. The approximation  of $f$ with integrands with the special structure  \eqref{specialcase}  allows us to apply the regularity result in \cite{el-mar-mas},  which ensures the local Lipschitz continuity of the approximating  minimizers.} 

Finally, we prove that the Lipschitz regularity of the approximating minimizers is preserved in passing to the limit and this implies the   Lipschitz regularity of  the local minimizer of the functional in \eqref{functional}.

The plan of the paper is briefly described. In Section \ref{sec2}, we give the precise assumptions (see (H1)--(H5)) and, in Section \ref{appsec}, we state the approximation results. The main idea is to generalize to $ (p,q)$-growth the approximation  result of Fonseca, Fusco and Marcellini \cite{fonsecafusco-marcellini} and utilize a suitable property of $f$ with respect to $x$,  similar  to the one considered  by Bildhauer and Fuchs \cite{bildfuchs} and by Esposito, Leonetti and Petricca in the framework of the Lavrentiev phenomenon in \cite{esp-leo-pet}.
Section \ref{resec} is devoted to the proof of the main result, by applying the Lipschitz regularity obtained by Euleuteri, Marcellini and Mascolo \cite{el-mar-mas} for the special structure \eqref{specialcase}. {Eventually, in Section \ref{sec5}, we show that the Lipschitz continuity result holds true also assuming strict convexity and growth conditions on $f$ only at infinity.}
We conclude exhibiting some examples of integral functionals with $(p,q)$-growth for which assumption (H5) holds true. We also show that  assumption (H5) is weaker than hypothesis (4) in \cite[Section 3]{esp-leo-pet} (see {\bf(iv)} in Section \ref{examples}).

\section{Assumptions and statement of the result} \label{sec2}

In this paper, we consider functional \eqref{functional}
where $f: \Omega \times \mathbb{R}^n\to [0,+\infty)$, $f=f(x,\xi)$, is strictly convex with respect to $\xi$ and satisfies the following assumptions:
\begin{itemize}
    \item[\bf{(H1)}] $f(x, \cdot) \in \mathcal{C}^2( \mathbb{R}^n)$, 
    \item[\bf{(H2)}] $f$ is \textit{$p$-uniformly convex}, that is there exist $p>1$ and $m>0$ such that  $$ m (1+|\xi|^2)^\frac{p-2}{2} |\lambda|^2 \le \sum_{i,j=1}^n f_{\xi_i\xi_j}(x,\xi)\lambda_i \lambda_j $$
    \item[\bf{(H3)}] there exists $M>0$ such that  $$|f_{\xi\xi}(x,\xi)| \le M (1+|\xi|^2)^\frac{q-2}{2},$$
    \item[\bf{(H4)}] there exist $q >p$ and $K>0$ such that 
    $$|f_{\xi x}(x,\xi)| \le K (1+|\xi|^2)^\frac{q-1}{2},$$
    \end{itemize}
    for almost every $x \in \Omega $ and for every $\xi,\lambda \in \mathbb{R}^n$. 
    
    Assume moreover that
    \begin{itemize}
    \item[\bf{(H5)}] for every $ \varepsilon >0 $ and every $x \in \Omega$ such that $B_{\varepsilon}(x) \Subset \Omega $, there exists $\tilde{y}=\tilde{y}(x,\varepsilon) $ such that with some function $c(\varepsilon) \ge 1$, $c(\varepsilon) \to 1$ as $\varepsilon \to 0$, we have $$f(\tilde{y},\xi) \le c(\varepsilon)f(y,\xi) \quad \text{for all } (y,\xi) \in  B_\varepsilon(x) \times \mathbb{R}^n .$$
\end{itemize}

\noindent We note that assumption (H4) implies that $f(x,\xi)=g(x,\xi)+f(x,0)$, for a function $g=g(x,\xi)$ verifying 
 \begin{itemize}
      \item[\bf{(H6)}] there exists $H>0$ such that for every $x,y \in \Omega$ and for every $\xi \in \mathbb{R}^n$, $$|g(x,\xi)-g(y,\xi)|\le H |x-y| (1+|\xi|^2)^\frac{q}{2}.$$
 \end{itemize}
For the proof of this result see Proposition \ref{proplip}. 
 From now on, without loss of generality, we will assume that $f$ satisfies (H6), i.e.\ $f(x,0) \equiv 0$; indeed, minimizers of  
 $$\mathcal{G}(u)= \int_{\Omega}  g(x,Du(x) )  dx$$
 are minimizers of $\mathcal{F}$.

{The main approximation result is the following}

\begin{thm}\label{thm1}
    Let $f$ satisfy (H1)--(H6)  with exponents $1<p<q$ such that 
    \begin{equation}\label{gap}
        q \le p \left( \dfrac{n+1}{n}  \right).
    \end{equation}
    Then, for every ball $B_R \Subset \Omega$, there exists a sequence of functions $f_h :B_R \times \mathbb{R}^n \to [0,+\infty)$, $f_h=f_h(x,\xi)$, with
    \begin{equation}\label{appf}
        f_h(x,\xi)= \sum_{i=1}^N a_{i,h}(x)g_{i,h}(\xi),
    \end{equation}
    satisfying (H1)--(H4) with constants $m$, $M$ and $K$ independent of $h$. 
    Moreover, $\xi \mapsto f_h(x,\xi)$ is strictly convex for all $x \in B_R$ and
    $$f_h(x,\xi) \to f(x,\xi) \ \text{uniformly in} \ B_R \times K, $$
for every $ K$ compact subset of $\mathbb{R}^n$. Furthermore, there exists a subsequence $(f_{h_k})$ such that for every $u \in W^{1,p}(B_R)$ such that $x \mapsto f(x,Du(x)) \in L^1(B_R)$
\begin{equation}\label{lim}
    \lim_k \int_{B_R} f_{h_k}(x,Du(x))dx = \int_{B_R} f(x,Du(x))dx.
\end{equation}

\end{thm}

We'd like to remark that
the assumptions (H1)--(H4) in Theorem \ref{thm1} allow us to prove the existence of the approximating sequence
\eqref{appf}. Our proof is inspired by the one of \cite[Theorem 2.7]{fonsecafusco-marcellini}, which concerns the case of standard growth conditions. On the other hand, the hypotheses (H5) and (H6) are used in the proof of Theorem \ref{thm1} in order to obtain \eqref{lim}.

We are able to derive the following Lipschitz regularity result.

\begin{thm}\label{thm2}
    Let $f $ satisfy (H1)--(H6) with exponents $1 < p < q$ such that \eqref{gap} is in force. Let $\bar{u} \in W^{1,p}(\Omega)$ be a local minimizer of \eqref{functional}. Then, $\bar{u} \in W^{1,\infty}_{loc}(\Omega)$ and the following estimate holds
    \begin{equation}\label{lipestimate}
      \Vert D \bar{u} \Vert_{L^\infty(B_\rho)} \le C \left[1+ \int_{B_R} f(x,D \bar{u})dx  \right]^\frac{\sigma}{p},  
    \end{equation}
     for every concentric balls $B_\rho \subset B_R \Subset \Omega$, with $C$ and $\sigma$ positive constants depending at most on $n,p,q,m,M,K,\rho$ and $R$.
\end{thm}

The previous theorem allows us to prove the Lipschitz continuity of local minimizers of \eqref{functional} assuming strict convexity and growth conditions on $f=f(x,\xi)$ only for large values of $|\xi|$. Precisely, we assume that $f$
satisfies (H1) and the assumptions:
     $$  m |\xi|^{p-2}|\lambda|^2 \le \sum_{i,j=1}^n f_{\xi_i\xi_j}(x,\xi)\lambda_i \lambda_j \leqno{\rm{{(H2)^*}}}$$
 $$|f_{\xi\xi}(x,\xi)| \le M |\xi|^{q-2} \leqno{\rm{{(H3)^*}}}$$
    $$|f_{\xi x}(x,\xi)| \le K |\xi|^{q-1} \leqno{\rm{{(H4)^*}}}$$
    $$|f(x,\xi)-f(y,\xi)|\le H |x-y| |\xi|^q. \leqno{\rm{{(H6)^*}}}$$
for every $x,y \in \Omega$ and for all $\xi,\lambda \in \mathbb{R}^n$, with $|\xi| \ge 1$ and for positive constants $m, M,K,H$.
We have the following regularity result.
\begin{thm}\label{thm3}
    Let $f=f(x,\xi)$ satisfy (H1), $\text{(H2)}^*-\text{(H4)}^*$, (H5) and $\text{(H6)}^*$ with exponents $1 < p < q$ such that \eqref{gap} is in force. Then, every local minimizer $\bar{u} \in W^{1,p}(\Omega)$ of \eqref{functional}
 is locally Lipschitz continuous. 
 \end{thm}

 \section{Notation and preliminary results}
 In this section, we introduce some notations and collect several results that we shall use to establish our main theorems.

We will follow  denote by $c$ or $C$ a general constant that may vary on different occasions, even within the same line of estimates. Relevant dependencies on parameters and special constants will be suitably emphasized using parentheses or subscripts.

In what follows, $B(x,r)=B_{r}(x)= \{ y \in \mathbb{R}^{n} : |y-x | < r  \}$ will denote the ball centered at $x$ of radius $r$. We shall omit the dependence on the center and on the radius when no confusion arises.

The next proposition shows that assumption (H4) implies the Lipschitz continuity of the partial map $x \mapsto f(x,\xi)$ up to an additive function $h=h(x)$.
\begin{prop}\label{proplip}
Let $f=f(x,\xi)$ satisfy assumption (H4). Then, $f(x,\xi)=g(x,\xi)+f(x,0)$, where the function $g=g(x,\xi)$ verifies (H6).
\end{prop}

\begin{proof} For the sake of simplicity, for $i=0,...,n-1$ and for every $\xi \in \mathbb{R}^n$, we will denote by $\xi_{-i}=(\xi_{i+1},...,\xi_n) \in \mathbb{R}^{n-i}$.
    From assumption (H4), we have for all $x,y \in \Omega$ and $\xi \in \mathbb{R}^n$
    $$|f_\xi(x,\xi)-f_\xi(y,\xi)|\le c|x-y| (1+|\xi|^2)^{\frac{q-1}{2}}$$
    which in particular yields
    $$|f_{\xi_1}(x,\xi)-f_{\xi_1}(y,\xi)|\le c|x-y| (1+|\xi|^2)^{\frac{q-1}{2}}.$$
    Integrating with respect to $\xi_1$, we get
    \begin{align*}
        &\left|\int_0^{\xi_1}  |f_{\xi_1}(x,t,\xi_2,...,\xi_n)-f_{\xi_1}(y,t,\xi_2,...,\xi_n)|dt  \right| \\
        & \,\,\,\,\,\,\,\,\,\, \le c|x-y| \left|   \int_0^{\xi_1} \left(1+|t|^2+\sum_{i=2}^n|\xi_i|^2\right)^{\frac{q-1}{2}} dt \right|\\
        & \,\,\,\,\,\,\,\,\,\, \le c|x-y| (1+|\xi|^2)^{\frac{q}{2}}.
    \end{align*}
This implies
$$ |g_0(x,\xi)-g_0(y,\xi)| \le c|x-y| (1+|\xi|^2)^{\frac{q}{2}}, $$
for all $x,y \in \Omega$ and all $\xi \in\mathbb{R}^n$, where we set $g_0(x,\xi):= f(x,\xi)-f(x,0,\xi_2,...,\xi_n)$.
\\
\noindent Now, note that $a_0:\Omega \times \mathbb{R}^{n-1} \to [0,+\infty)$, $a_0(x,\eta):= f(x,0,\eta_1,...,\eta_{n-1})$, $\eta =(\eta_1,...,\eta_{n-1})$, satisfies assumption (H4). Hence, arguing as above, we infer
$$ |g_1(x,\eta)-g_1(y,\eta)| \le c|x-y| (1+|\eta|^2)^{\frac{q}{2}}, $$
for all $x,y \in \Omega$ and all $\eta \in \mathbb{R}^{n-1}$, where we define $g_1(x,\eta):= a_0(x,\eta)-a_0(x,0,\eta_2,...,\eta_{n-1})$. At this stage, we can write for all $x \in \Omega$ and every $\xi \in \mathbb{R}^n$,
$$f(x,\xi)= g_0(x,\xi)+g_1(x,\xi_{-1})+a_1(x,\xi_{-2}),$$
where the functions $g_0$ and $g_1$ satisfy assumption (H4) and we set $a_1(x,\xi_{-2}):= f(x,0,0,\xi_3,...,\xi_{n})$. 
Finally, we find that
$$f(x,\xi)= \sum_{i=0}^{n-1}g_i(x,\xi_{-i})+f(x,0),$$
where the functions $g_i$ satisfy assumption (H4).
\end{proof}

{
If $f=f(x,\xi)$ satisfies assumption (H4) only for large values of $|\xi|$, we have the following
\begin{prop}
   Let $f=f(x,\xi)$ satisfy assumption $(H4)^*$. 
Then, $f(x,\xi)=\tilde{g}(x,\xi)+l(x)$, for a function $\tilde{g}=\tilde{g}(x,\xi)$ verifying $(H6)^*$.   
\end{prop}
\begin{proof}
  Assumption $(H4)^*$ implies that for all $x,y \in \Omega$ and every $|\xi| \ge 1$
    $$|f_{\xi_1}(x,\xi)-f_{\xi_1}(y,\xi)|\le c|x-y| (1+|\xi|^2)^{\frac{q-1}{2}}.$$ 
     Integrating with respect to $\xi_1$, we get
    \begin{align*}
        &\left|\int_1^{\xi_1}  |f_{\xi_1}(x,t,\xi_2,...,\xi_n)-f_{\xi_1}(y,t,\xi_2,...,\xi_n)|dt  \right| \\
        & \,\,\,\,\,\,\,\,\,\, \le c|x-y| \left|   \int_1^{\xi_1} \left(1+|t|^2+\sum_{i=2}^n|\xi_i|^2\right)^{\frac{q-1}{2}} dt \right|\\
        & \,\,\,\,\,\,\,\,\,\, \le \tilde{c}|x-y| (1+|\xi|^2)^{\frac{q}{2}}.
    \end{align*}
This implies
\begin{equation}\label{H4star}
    |d(x,\xi)-d(y,\xi)| \le \tilde{c}|x-y| (1+|\xi|^2)^{\frac{q}{2}},
\end{equation}
for all $x,y \in \Omega$ and all $|\xi| \ge 1 $, where we set $d(x,\xi):= f(x,\xi)-f(x,1,\xi_2,...,\xi_n)$. 
\\
Now, we observe that $\tilde{f}(x,\eta):= f(x,1,\eta_1,...,\eta_{n-1})$, $\eta =(\eta_1,...,\eta_{n-1}) \in \mathbb{R}^{n-1}$, satisfies assumption (H4). Therefore, Proposition \ref{proplip} applied to the function $\tilde{f}$, together with \eqref{H4star}, gives the conclusion.
\end{proof}
}

{
\begin{rmk}\label{rmk1}
    By assumptions (H2) and (H6) there exist positive constants $c$, $c_\Omega$ such that for all $x \in \Omega$ and for every $\xi \in \mathbb{R}^n$
    $$f(x,\xi) \ge c |\xi|^p-c_\Omega.  $$
    Indeed, for any $\xi, \eta \in \mathbb{R}^n$, using Taylor's formula, assumption (H2) and \cite[Lemma 8.3]{Giusti}, we have
    \begin{align*}
        f(x,\xi)  \ge & \ f \left( x, \dfrac{\xi+\eta}{2} \right) + \dfrac{1}{2}\langle f_\xi(x,\xi), \xi-\eta \rangle \\
        & \ + c \left| \dfrac{\xi-\eta}{2} \right|^2 \int_0^1 (1-t) \left( 1+ \left| \dfrac{\xi+\eta}{2} +t \left( \dfrac{\xi-\eta}{2} \right) \right|^2 \right)^\frac{p-2}{2} dt\\
        \ge & \ f \left( x, \dfrac{\xi+\eta}{2} \right) + \dfrac{1}{2}\langle f_\xi(x,\xi), \xi-\eta \rangle \\
         & \ + c \left| \dfrac{\xi-\eta}{2} \right|^2  \left( 1+ \left| \dfrac{\xi+\eta}{2}\right|^2 +\left|  \dfrac{\xi-\eta}{2}  \right|^2 \right)^\frac{p-2}{2} .
    \end{align*}
    Similarly, we derive
    \begin{align*}
        f(x,\eta)  \ge & \ f \left( x, \dfrac{\xi+\eta}{2} \right) + \dfrac{1}{2}\langle f_\xi(x,\xi), \eta-\xi \rangle \\
         & \ + c \left| \dfrac{\xi-\eta}{2} \right|^2  \left( 1+ \left| \dfrac{\xi+\eta}{2}\right|^2 +\left|  \dfrac{\xi-\eta}{2}  \right|^2 \right)^\frac{p-2}{2} .
    \end{align*}
    Summing up the previous inequalities, we get
    \begin{align}
        f(x,\xi) +f(x,\eta) \ge & \ 2 f \left( x, \dfrac{\xi+\eta}{2} \right) + \dfrac{1}{2}\langle f_\xi(x,\xi)-f_\xi(x,\eta), \xi-\eta \rangle \notag \\
         & \ + c \left| \dfrac{\xi-\eta}{2} \right|^2  \left( 1+ \left| \dfrac{\xi+\eta}{2}\right|^2 +\left|  \dfrac{\xi-\eta}{2}  \right|^2 \right)^\frac{p-2}{2} . \label{pqg}
    \end{align}
    Notice that (H2) implies
    $$\langle f_\xi(x,\xi)-f_\xi(x,\eta), \xi-\eta \rangle \ge c (1+|\xi|^2+|\eta|^2)^\frac{p-2}{2}|\xi-\eta|^2$$
    for any $\xi,\eta \in \mathbb{R}^n$ and for all $x\in \Omega$. Therefore, the second term in the right-hand side of \eqref{pqg} can be discarded, since it is non-negative. Now, choosing $\eta=0$ and recalling that $f$ is non-negative, we find that
    $$f(x,\xi)+f(x,0) \ge c |\xi|^p.$$
    On the other hand, for a fixed $y \in \Omega$, by assumption (H6), we have that for all $x \in \Omega$
    $$f(x,0) \le f(y,0)+ H |x-y| \le f(y,0)+ H \rm{diam}(\Omega) =:c_\Omega< + \infty,$$
    since $\Omega $ is bounded. 
\end{rmk}
}



The special structure of the approximating energy densities in \eqref{appf} allows us to   use  the following  result (\cite[Theorem 5.1]{el-mar-mas}) that implies  the Lipschitz regularity of the approximating minimizers.

\begin{thm}\label{thmEMM}
    Let
    $$b(x,\xi)= \sum_{i=1}^N a_i(x) g_i(\xi),$$
    where $a_i(x)>0$ in $\Omega$, $a_i \in W^{1,r}(\Omega)$, $r >n$, and $g_i : \mathbb{R}^n \to [0,+\infty)$ is strictly convex in $\xi$. Moreover, assume that $b=b(x,\xi)$ satisfies (H1)--(H3) and that (H4) holds with $K(x)=\sum_{i=1}^N|D_x a_i| \in L^r(\Omega)$ and 
    \begin{equation}\label{gapEMM}
     \dfrac{q}{p} < 1+\dfrac{1}{n}-\dfrac{1}{r}.   
    \end{equation}
    Then, every local minimizer $\bar{v} \in W^{1,p}(\Omega)$ of the integral functional
    $$\int_{\Omega}b(x,Dv)dx= \sum_{i=1}^N\int_{\Omega}a_i(x)g_i(Dv)dx$$
    is locally Lipschitz continuous and the following estimate holds
    \begin{equation}\label{lipest}
        \Vert D \bar{v} \Vert_{L^\infty(B_\rho)} \le C \Vert 1+ K \Vert_{L^r(B_R)}^{\beta \gamma} \left[ 1+ \int_{B_R} b(x,D \bar{v})dx \right]^\frac{\gamma}{p}
    \end{equation}
    for every concentric balls $B_\rho \subset B_R \Subset \Omega$, where $C$, $\beta$ and $\gamma$ are positive constants depending at most on $n,r,p,q,m,M,\rho$ and $R$.
\end{thm}

\section{An Approximation Result}\label{appsec}
This section is devoted to the construction of the approximating sequence of functionals under $(p,q)$-growth, with a procedure inspired by \cite{fonsecafusco-marcellini} in the case of standard growth.

\begin{proof}[Proof of Theorem \ref{thm1}]
Let $B_{R} \Subset \Omega$ be a ball and let $\psi \in \mathcal{C}^\infty_0(\mathbb{R}^n)$ be a cut-off function such that $0 \le \psi(x) \le 1$ for all $x$, $\text{supp } \psi \subset (-3,3)^n$ and such that $\psi(x) \equiv 1$ if $x \in [-1,1]^n$. For any $h \in \mathbb{N}$ we denote by $Q_{i,h}(x_{i,h})$ the standard covering of $\mathbb{R}^n$ with
closed cubes, centered at $x_{i,h}$, with sides parallel to the coordinates axes, side length equal to $2/h$ and having
pairwise disjoint interiors. Then, for any $i, h$, we set $$\psi_{i,h}(x):= \psi(h(x-x_{i,h})) $$
and
$$\sigma_h(x):= \sum_{i=1}^\infty \psi_{i,h}(x), \qquad \qquad \varphi_{i,h}(x):= \dfrac{\psi_{i,h}(x)}{\sigma_h(x)}.$$
Finally, for all $h$ such that $12 \sqrt{n}/h < \text{dist}(B_R; \partial \Omega)$, and for every $x \in B_R$, $\xi \in \mathbb{R}^n$, we define
$$f_h(x,\xi):= \sum_{i=1}^\infty \varphi_{i,h}(x) f(x_{i,h},\xi).$$
Notice that the above sum is finite and consists of at most $3^n$ terms. \\ 
\noindent \textbf{Step 1:} It is easy to check that each $f_h$ satisfies (H1)--(H3). Now, we prove that (H4) holds true. Let us fix $x_0 \in B_R$ and $\xi \in \mathbb{R}^n$.
By construction there exist at most $3^n$ cubes, $Q_{j_1,h}(x_{j_1,h}),...,Q_{j_{3^n},h}(x_{j_{3^n},h})$, such that for any $x$ in a suitable neighborhood $U$ of $x_0$ we can write
$$f_h(x,\xi)= \sum_{l=1}^{3^n} \varphi_{j_l,h}(x) f(x_{j_l,h},\xi) \quad \text{with} \quad \sum_{l=1}^{3^n} \varphi_{j_l,h}(x)=1.$$
Since $\sum_{l=1}^{3^n} \varphi_{j_l,h}(x)=1$, we have
$\varphi_{j_1,h}(x)=1-\sum_{l=2}^{3^n} \varphi_{j_l,h}(x)$. This gives that
$$ D_{x}\varphi_{j_1,h}(x)=-\sum_{l=2}^{3^n} D_{x}\varphi_{j_l,h}(x).$$
Therefore for all $x \in U$ we have that
\begin{equation}\label{der}
    D_{\xi x}f_h(x,\xi)= \sum_{l=1}^{3^n} D_x\varphi_{j_l,h}(x) D_\xi f(x_{j_l,h},\xi)= \sum_{l=2}^{3^n} D_x\varphi_{j_l,h}(x)\left[  D_\xi f(x_{j_l,h},\xi) -  D_\xi f(x_{j_1,h},\xi) \right].
\end{equation}
Observe that, since $(Q_{j_l,h}(x_{j_l,h}))_{j=1}^{3^n}$ is a finite family of cubes centered at $x_{j_l,h}$ with side length equal to $2/h$,
$$|x_{j_l,h}-x_{j_1,h}| \le \dfrac{c(n)}{h}.$$
By virtue of assumption (H4), we have that for all $l$
$$|D_\xi f(x_{j_l,h},\xi) - D_\xi f(x_{j_1,h},\xi) | \le \dfrac{c(n)K}{h} (1+|\xi|^2)^\frac{q-1}{2}.$$
On the other hand, for any $j$, there exists a set of indices $I_j$ containing $j$, with $\#(I_j)=3^n$, such that for all $x \in \mathbb{R}^n$
$$D_{x} \varphi_{j,h}(x)= \dfrac{D_{x}\psi_{j,h}(x)}{\sigma_h(x)}-\dfrac{\psi_{j,h}(x)}{\sigma^2_h(x)} \sum_{k \in I_j} D_x \psi_{k,h}(x).$$
Hence, since by construction $\sigma_h(x)  \ge 1$ for all $x$, we have that, for every $x \in \mathbb{R}^n$ and for all $j \in \mathbb{N}$,
$$|D_x \varphi_{j,h}(x)| \le (3^n+1)h \Vert D_x \psi \Vert_{\infty}.$$
In view of the above estimates and from \eqref{der}, we can conclude that for all $(x,\xi) \in B_R \times \mathbb{R}^n$ and for any $h$
$$| (f_h)_{\xi x} (x,\xi)| \le \tilde{c}(n) K \Vert D_x \psi \Vert_{\infty} (1+|\xi|^2)^\frac{q-1}{2}$$
and thus (H4) holds.

{
Now, fix a positive constant $M$. Then, for a.e.\ $x \in \Omega$ and every $\xi \in \mathbb{R}^n$ such that $|\xi| \le M$, we have
\begin{align*}
    |f_h(x,\xi)-f(x,\xi)| \le & \sum_{i=1}^N \varphi_{i,h}(x) |f(x_{i,h},\xi)-f(x,\xi)| \\
    \le & \sum_{i=1}^N \varphi_{i,h}(x) |x_{i,h}-x|(1+|\xi|^2)^\frac{q}{2}\\
    \le & \dfrac{1}{h} N (1+M^2)^\frac{q}{2}
\end{align*}
that goes to $0$ as $h \to + \infty$.
}
This proves that $f_h(x,\xi) \to f(x,\xi)$ uniformly in $B_R \times K$ for every $K$ compact subset of $\mathbb{R}^n$.  

\noindent \textbf{Step 2:} Assumptions (H5) and (H6) allows us to prove \eqref{lim}. Observe that, from (H3) it follows that
$$ |f(x,\xi)|\le C(M) (1+|\xi|^2)^{\frac{q}{2}}.$$ Hence, for every $v \in W^{1,q}(B_R)$, we have
\begin{equation}\label{conv}
\lim_h \int_{B_R} f_h(x,Dv(x))dx = \int_{B_R}f(x,Dv(x))dx,
\end{equation}
since $f_h(x,Dv(x)) \to f(x,Dv(x))$ a.e.\ in $B_R$. The growth assumption on $f$ yields
$$|f_h(x,Dv(x))| \le c (1+|Dv(x)|^2)^\frac{q}{2}. $$
and so  \eqref{conv} follows, by the use of Lebesgue's dominated convergence Theorem.

For a sequence $\varepsilon_k\downarrow 0$, let $\Phi_{\varepsilon_k}$ be a standard sequence of mollifiers.  For a fixed $u \in W^{1,p}(B_R)$ such that $x \mapsto f(x,Du(x)) \in L^1(B_R)$, let us  consider $u_{\varepsilon_k}(x)=(u * \Phi_{\varepsilon_k})(x)$ the usual mollification. Then, we observe 
\begin{equation}\label{grad}
    |Du_{\varepsilon_k} (x)| \le \left(  \int_{\Omega} |Du(y)|^p dy\right)^\frac{1}{p}  \left(  \int_{\Omega} |\Phi_{\varepsilon_k}(y)|^{p'} dy\right)^\frac{1}{p'} \le c \varepsilon^{-\frac{n}{p}}_k,
\end{equation}
where the constant $c$ depends on $\Vert Du \Vert_{L^p(\Omega)}$.

 Now, for  $x \in B_R$ and $\varepsilon_k$ small enough to have $B_{\varepsilon_k}(x) \Subset \Omega$, we choose $\tilde{x} \in B_{\varepsilon_k}(x)$ according to assumption (H5). By virtue of assumption (H6), we have
\begin{align*}
    f(x,Du_{\varepsilon_k}(x)) & \le f(\tilde{x},Du_{\varepsilon_k}(x))+ (1+|Du_{\varepsilon_k}(x)|^2)^\frac{q}{2}\varepsilon_k \\
    & \le f(\tilde{x},Du_{\varepsilon_k}(x))+ c\varepsilon_k^{-\frac{nq}{p}+1} 
\end{align*}
where in the last line we used  \eqref{grad}. Integrating over the ball $B_R$ with respect to $x$, from Jensen's inequality and assumption (H5), it follows 
\begin{align*}
    \int_{B_R} f(x,Du_{\varepsilon_k}(x)) dx
    \le & \int_{B_R}\int_{|x-y|<\varepsilon_k} f(\tilde{x},Du(y)) \Phi_{\varepsilon_k}(x-y)dydx+ c\varepsilon_k^{-\frac{nq}{p}+1+n} \\
    \le & c(\varepsilon_k) \int_{B_R}\int_{|x-y|<\varepsilon_k} f(y,Du(y)) \Phi_{\varepsilon_k}(x-y)dydx+ c\varepsilon_k^{-\frac{nq}{p}+1+n}\\
    = & {c(\varepsilon_k) \int_{B_R}(f(x,Du(x)))_{\varepsilon_k}dx+ c\varepsilon_k^{-\frac{nq}{p}+1+n}},
\end{align*}
{where $(f(\cdot,Du(\cdot)))_{\varepsilon_k}$ denotes the regularized function of $f(x,Du(x))$.}
We note that, by assumption \eqref{gap}, it holds $-\frac{nq}{p}+1+n>0$. Therefore, letting $k \to \infty$ in both sides of previous estimate,
\begin{equation*}
\lim_k \int_{B_R} f(x,Du_{\varepsilon_k}(x))dx \le {\lim_k} \int_{B_R}(f(x,Du(x)))_{\varepsilon_k}dx,
\end{equation*}
{where we used that $c(\varepsilon_k)\to 1$ as $\varepsilon_k\to 0$.}
Then, since {$f(x,Du(x))\in L^1(B_R)$ we have that} $(f(\cdot,Du(\cdot)))_{\varepsilon_k} \to f(x,Du(x))$ strongly in $L^1(B_R)$, and so
\begin{equation*}
\lim_k \int_{B_R} f(x,Du_{\varepsilon_k}(x))dx \le \int_{B_R}f(x,Du(x))dx.
\end{equation*}
Since $(u_{\varepsilon_k})$ converges to $u$ strongly in $W^{1,p}(B_R)$, by weak lower semicontinuity,
\begin{equation*}
\int_{B_R}f(x,Du(x))dx \le \lim_k \int_{B_R} f(x,Du_{\varepsilon_k}(x))dx. 
\end{equation*}
Hence, for $u \in W^{1,p}(B_R)$ such that $x \mapsto f(x,Du(x)) \in L^1(B_R)$ we find
\begin{equation}\label{conv1}
    \lim_k \int_{B_R} f(x,Du_{\varepsilon_k}(x))dx= \int_{B_R}f(x,Du(x))dx .
\end{equation}
Thanks to \eqref{conv}, for every fixed $k \in \mathbb{N}$ we get
\begin{equation*}
\lim_h \int_{B_R} f_h(x,Du_{\varepsilon_k}(x))dx = \int_{B_R}f(x,Du_{\varepsilon_k}(x))dx.
\end{equation*}
Therefore, for $k=1$ there exists $h_1$ such that
$$\left\vert \int_{B_R}  f_{h_1}(x,Du_{\varepsilon_1}(x)) - \int_{B_R}  f(x,Du_{\varepsilon_1}(x))  dx \right\vert < \dfrac{1}{2}.$$
For $k=2$ there exists $h_2>h_1$ such that
$$\left\vert \int_{B_R}  f_{h_2}(x,Du_{\varepsilon_2}(x)) - \int_{B_R}  f(x,Du_{\varepsilon_2}(x))  dx \right\vert < \dfrac{1}{2^2}.$$
Iterating this procedure  for every  $k\in \mathbb{N}$, we get that there exists $h_k > h_{k-1} >. \, . \, .>h_1$ such that
$$\left\vert \int_{B_R}  f_{h_k}(x,Du_{\varepsilon_k}(x)) - \int_{B_R}  f(x,Du_{\varepsilon_k}(x))  dx \right\vert < \dfrac{1}{2^k},$$
which implies, together with \eqref{conv1}, 
\begin{equation*}
\lim_k \int_{B_R} f_{h_k}(x,Du_{\varepsilon_k}(x))dx = \lim_k \int_{B_R} f(x,Du_{\varepsilon_k}(x))dx= \int_{B_R}f(x,Du(x))dx,
\end{equation*}
i.e.\ the conclusion \eqref{lim}.
\end{proof}

\section{Proof of Theorem \ref{thm2}.}\label{resec}

Combining the approximation given by Theorem \ref{thm1} with the regularity result of Theorem \ref{thmEMM}, we are  in a position to give the

\begin{proof}[Proof of Theorem \ref{thm2}]
Let $\bar u\in W^{1,p}_{\mathrm{loc}}(\Omega)$ be a local minimizer of the functional at \eqref{functional}. 
Consider the boundary value problem in $B_R \Subset \Omega$
\begin{equation}\label{vp}
    \inf \left\{ \int_{B_R} f_h(x,Dv)dx \, : \, v \in W^{1,p}_0(B_R)+ \bar{u} \right\},
\end{equation}
where $f_h=f_h(x,\xi)$ is the sequence given by Theorem \ref{thm1}. 

By the strict convexity of $f_h$, there exists a unique minimum $v^h \in W^{1,p}(B_R)$ of the problem \eqref{vp}. Note that $f_h$ satisfy the assumptions of Theorem \ref{thmEMM} with $K>0$ constant and $r=+\infty$. In this case  assumption \eqref{gapEMM} reduces to \eqref{gap} and  we have that $v^h\in W^{1,\infty}_{loc}(B_R)$ with the estimate 
$$\Vert D v^h \Vert_{L^\infty(B_\rho)} \le C (1+K)^{\beta \gamma} \left[ 1+ \int_{B_r} f_h(x,D v^h)dx \right]^\frac{\gamma}{p},$$
for concentric balls $B_\rho \subset B_r \subset B_R$, with constants $C,\beta$ and $\gamma$ independent of $h$. By the minimality of $v^h$, it follows
\begin{equation*}
    \Vert D v^h \Vert_{L^\infty(B_\rho)} \le C (1+ K )^{\beta \gamma} \left[ 1+ \int_{B_R} f_h(x,D \bar{u})dx \right]^\frac{\gamma}{p}.
\end{equation*}
   Since {$f(x,D\bar{u})\in L^{1}(B_R)$},   by the property of $f_h$ at \eqref{lim} and the minimality of $\bar{u}$, we get
   \begin{equation}\label{limconv}
       \lim_h \int_{B_R} f_h(x,D \bar{u})dx = \int_{B_R} f(x,D \bar{u})dx { \ < +\infty},
   \end{equation}
 and so, 
   \begin{equation}\label{linfty}
    \Vert D v^h \Vert_{L^\infty(B_\rho)} \le C (1+ K )^{\beta \gamma} \left[ 1+ \int_{B_r} f(x,D \bar{u})dx \right]^\frac{\gamma}{p}. 
\end{equation} 
Therefore, up to subsequences, $(v^h)$ converges weakly* in $W^{1,\infty}_{\mathrm{loc}}(B_R)$  as $h \to \infty$ to a function $\bar{v} \in W^{1,\infty}_0(B_R)+ \bar{u}$
{such that
\begin{equation}\label{linftylim}
    \Vert D \bar{v} \Vert_{L^\infty(B_\rho)} \le C (1+ K )^{\beta \gamma} \left[ 1+ \int_{B_R} f(x,D \bar{u})dx \right]^\frac{\gamma}{p}.
\end{equation}} 

\noindent By the lower semicontinuity of the functional at \eqref{functional} in $W^{1,p}(B_R)$, the minimality of $v^h$ and the equality at \eqref{limconv}, we have
\begin{align*}
   & \int_{B_\rho} f(x, D \bar{v} )dx \le \liminf_h \int_{B_\rho} f(x, D v^h )dx \\
   & \,\,\,\,\,\,\,\,\,\, = \liminf_h \int_{B_\rho} f_h(x, D v^h )dx  \le \liminf_h \int_{B_R} f_h(x, D \bar{u})dx\\
   & \,\,\,\,\,\,\,\,\,\,= \int_{B_R} f(x, D \bar{u} )dx,
\end{align*}
where the first equality in the preceding formula follows from the uniform convergence in compact sets of $f_h \to f$, since $\Vert Dv^h \Vert_{L^\infty(B_\rho)} \le L$. 
\\
\noindent Letting $\rho \to R$,
\begin{equation} \label{intineq}\int_{B_R} f(x, D \bar{v} )dx \le \int_{B_R} f(x, D \bar{u} )dx,
\end{equation}
which implies $\bar{v}=\bar{u}$ for the strict convexity of $f$. 
Then $\bar{u} \in W^{1,\infty}_{loc}(B_R)$ and estimate \eqref{lipestimate} holds.
  \end{proof}


\section{Strict convexity only at infinity}\label{sec5}
This section is devoted to the proof of Theorem \ref{thm3}. We argue as in \cite{EMM1,fonsecafusco-marcellini}.
\begin{proof}[Proof of Theorem \ref{thm3}]
For  $k \in \mathbb{N}$, let us  define the following sequence 
$$f_k(x,\xi):= f(x,\xi)+\dfrac{1}{k} (1+|\xi|^2)^{\frac{q}{2}},$$
and, for a fixed ball $B_R \Subset \Omega$, let us consider the problem
\begin{equation}\label{vp1}
    \inf \left\{ \int_{B_R} f_k(x,Dv)dx \, : \, v \in W^{1,q}_0(B_R)+ \bar{u}_{\varepsilon} \right\},
\end{equation}
where $\bar{u}_\varepsilon$ is the usual mollification of $\bar{u}$. Since $f_k$ satisfies $q$-growth conditions, there exists a unique solution $v_k^\varepsilon \in W^{1,q}_0(B_R)+ \bar{u}_{\varepsilon} $ of the problem \eqref{vp1}. On the other hand, $f_k$ satisfies hypotheses (H1)--(H6), therefore, by Theorem \ref{thm2} with $p=  q$, we have $v^\varepsilon_k  \in W^{1,\infty}_{loc}(B_R)$ and the a priori estimate
 \begin{equation}\label{lipestimate1}
      \Vert D v^\varepsilon_k \Vert_{L^\infty(B_\rho)} \le C \left[1+ \int_{B_R} f_k(x,D v^\varepsilon_k)dx  \right]^\frac{\sigma}{p},  
    \end{equation}
    holds for every concentric balls $B_\rho \Subset B_R $, with $C$ and $\sigma$ positive constants depending at most on $n,p,q,m,M,K,\rho$.
Recalling the definition of $f_k$ and using the minimality of $v^\varepsilon_k$, by estimate \eqref{lipestimate1} we get
     \begin{equation*}
      \Vert D v^\varepsilon_k \Vert_{L^\infty(B_\rho)} \le C \left[1+ \int_{B_R} f(x,D \bar{u}_\varepsilon)dx + \dfrac{1}{k} \int_{B_R} (1+|D \bar{u}_\varepsilon|^2)^\frac{q}{2} dx  \right]^\frac{\sigma}{p}. 
    \end{equation*}
    Since $\bar{u}_\varepsilon \in W^{1,q}(B_R)$, we conclude that
    $$v^\varepsilon_k \rightharpoonup v^\varepsilon \quad \text{weakly* in } W^{1,\infty}(B_\rho) \text{ as } k \to  \infty  $$
     and by the previous estimate
     \begin{equation*}
       \Vert D v^\varepsilon \Vert_{L^\infty(B_\rho)} \le \liminf_{k}\Vert D v^\varepsilon_k \Vert_{L^\infty(B_\rho)} \le C \left[1+ \int_{B_R} f(x,D \bar{u}_\varepsilon)dx  \right]^\frac{\sigma}{p}. 
    \end{equation*} 
Now, we observe that from (H2)*, we have
$$|\xi|^p \le c(1+f(x,\xi)), \quad \text{for every } \xi \in \mathbb{R}^n \text{ and all } x \in \Omega .$$
Thus, we obtain
\begin{align*}
    \int_{B_{{R}}} |Dv^\varepsilon_k|^p dx \le & C \int_{B_{{R}}} (1+ f(x,Dv^\varepsilon_k)) dx \\
    \le & C \int_{B_{{R}}} (1+ f(x,Dv^\varepsilon_k)) dx + \dfrac{1}{k} \int_{B_{{R}}} (1+|Dv^\varepsilon_k|^2)^\frac{q}{2}dx \\
    \le & C \int_{B_{{R}}} (1+ f(x,D \bar{u}_\varepsilon)) dx + \dfrac{1}{k} \int_{B_{{R}}} (1+|D \bar{u}_\varepsilon|^2)^\frac{q}{2}dx
\end{align*}
with a constant $C$ independent of $k$. Therefore,
$$v^\varepsilon_k \rightharpoonup v^\varepsilon \quad \text{weakly in } W^{1,p}(B_{{R}}) \text{ as } k \to  \infty  $$
and 
$$v^\varepsilon_k-\bar{u}_\varepsilon \rightharpoonup v^\varepsilon-\bar{u}_\varepsilon \quad \text{weakly in } W^{1,p}_{0}(B_{{R}}) \text{ as } k \to  \infty.  $$
By weak lower semicontinuity, the minimality of $v_k^\varepsilon$ and \eqref{conv1}, we have
\begin{align*}
    \int_{B_R} f(x,Dv^\varepsilon)dx \le  &\liminf_k \int_{B_R} f(x,Dv_k^\varepsilon)dx \\
    \le & \int_{B_R} f(x,D\bar{u}_\varepsilon)dx\\
    \le & C \left( 1+ \int_{B_R} f(x,D\bar{u})dx \right)
\end{align*}
where $C$ is independent of $\varepsilon$. 
\\
Arguing as above, we find that there exists $\bar{v} \in W^{1,p}_0(B_R) + \bar{u} $ such that $v^\varepsilon - \bar{u}_\varepsilon \rightharpoonup \bar{v} -\bar{u} $ weakly in $W^{1,p}_0(B_R)$ as $\varepsilon \to 0$ and the a priori estimate
\begin{equation}\label{lipv}
       \Vert D \bar{v} \Vert_{L^\infty(B_\rho)} \le C \left[1+ \int_{B_R} f(x,D \bar{u})dx  \right]^\frac{\sigma}{p}. 
    \end{equation}
    holds. Moreover, $\bar{v}$ is a solution to \eqref{functional}. Indeed, the semicontinuity of the functional at \eqref{functional}, the minimality of ${v^{\varepsilon}_k}$ and \eqref{conv1} yield
\begin{align*}
    \int_{B_R} f(x,Dv)dx \le  &\liminf_\varepsilon \int_{B_R} f(x,Dv^\varepsilon)dx 
    \\\le& \liminf_{\varepsilon}{\liminf_k \int_{B_R} f(x,Dv_k^\varepsilon)dx}\\
    \le & \liminf_\varepsilon \int_{B_R} f(x,D\bar{u}_\varepsilon)dx\\
    = & \int_{B_R} f(x,D\bar{u})dx .
\end{align*}  
Since $f=f(x,\xi)$ is not strictly convex for $|\xi|>0$, we may not conclude that $\bar{u}=\bar{v}$ in $B_R$. Set
$$ \Sigma = \left\{  x \in B_R : \left|  \dfrac{D \bar{u}(x)+D \bar{v}(x)}{2} \right| >1 \right\} \quad \text{and} \quad w=\dfrac{\bar{u}+\bar{v}}{2}.$$
If $\Sigma$ has positive measure, then from the convexity of $f(x,\cdot)$, we have
\begin{equation}\label{eq1}
    \int_{B_R \setminus \Sigma} f(x,Dw)dx \le \dfrac{1}{2} \int_{B_R \setminus \Sigma} f(x,D \bar{u})dx + \dfrac{1}{2} \int_{B_R \setminus \Sigma} f(x,D \bar{v})dx.
\end{equation}
Now, the scrict convexity of $f(x,\xi)$ for $|\xi| \ge 1$ gives that
$$f(x,D\bar{u}) > f(x,Dw) + \langle f_\xi(x,Dw), D \bar{u}-Dw \rangle$$
and
$$f(x,D\bar{v}) > f(x,Dw) + \langle f_\xi(x,Dw), D \bar{v}-Dw \rangle.$$
By adding up the previous two inequalities, we have
\begin{equation}\label{eq2}
    \int_{B_R \cap \Sigma} f(x,Dw)dx < \dfrac{1}{2} \int_{B_R \cap \Sigma} f(x,D \bar{u})dx + \dfrac{1}{2} \int_{B_R \cap \Sigma} f(x,D \bar{v})dx.
\end{equation}
Adding \eqref{eq1} and \eqref{eq2}, we get a contradiction with the minimality of $\bar{u}$ and $\bar{v}$. Therefore, the set $\Sigma$ has zero measure, which implies that
$$\Vert D \bar{u} \Vert_{L^\infty(B_\rho)} \le 2 + \Vert D \bar{v} \Vert_{L^\infty(B_\rho)}$$
and so, by \eqref{lipv}, $\bar{u} \in W^{1,\infty}_{loc}(\Omega)$.
\end{proof}

\section{Some Examples}\label{examples}

In this section, we give some examples of functions for which Theorem \ref{thm2} holds.

\begin{itemize}
    \item[\bf{(i)}] Let $f=f(x,\xi)$ satisfy (H1)--(H4) and (H6) with $p=q$. Moreover, assume that there exist positive constants $c_0$ and $c_1$ such that
$$c_0 (1+|\xi|^2)^\frac{p}{2} \le f(x,\xi) \le c_1 (1+|\xi|^2)^\frac{p}{2}.$$
 Fix $ \varepsilon >0 $ and $\bar{x} \in \Omega$ such that $B_{\varepsilon}(x) \Subset \Omega $. Then, for every $x \in B(\bar{x},\varepsilon) $, assumption (H6) gives that
$$f(\bar{x},\xi) \le f(x,\xi)+\varepsilon H(1+|\xi|^2)^\frac{p}{2} \le {\tilde{c}(1+\varepsilon)f(x,\xi)}.$$
Therefore, $f$ satisfies (H5) with $c(\varepsilon):= \tilde{c}(1+\varepsilon)$.

\item[\bf{(ii)}] Consider  
$$f(x,\xi)=a_1(x)f_1(\xi)+a_2(x)f_2(\xi)$$
with the assumptions $f_1,f_2 \ge 0$, $a_1,a_2 \ge 0$ and $ a_1 \ge C_1>0$, $a_2 \in \mathcal{C}^{0,1}(\Omega)$. Functions with this special structure were first considered in \cite{el-mar-mas}. For $a_2 \equiv 1$, we have the energy function considered in \cite{BCM,colmin,colmin2,zhikov1}.

Let $B(x,\varepsilon) \Subset \Omega$ and $x_0 \in \overline{B(x,\varepsilon)}$ such that $a_2(x_0) \le a_2(y)$, for every $y \in {B(x,\varepsilon)}$. Then,
\begin{align*}
    f(x_0, \xi) = & \ a_1(x_0) f_1(\xi) + a_2(x_0) f_2(\xi) \\
    \le & \ [a_1(x_0)-a_1(y)]f_1(\xi) +a_1(y)f_1(\xi)+a_2(x_0)f_2(\xi) \\
    \le & \ [a_1(x_0)-a_1(y)]f_1(\xi) +a_1(y)f_1(\xi)+a_2(y)f_2(\xi)
\end{align*}
for all $y \in {B(x,\varepsilon)}$ and every $\xi \in \mathbb{R}^n$. Moreover, the Lipschitz continuity of $a_1$ yields
$$|a_1(x_0)-a_1(y)| \le L_1 |x_0-y|< L_1 \varepsilon.$$
Combining the previous inequalities, we find that
\begin{align*}
    f(x_0, \xi) 
    \le & \ L_1\varepsilon f_1(\xi) +a_1(y)f_1(\xi)+a_2(x_0)f_2(\xi) \\
    \le & \ L_1\varepsilon \dfrac{a_1(x)}{C_1}f_1(\xi)+f(x,\xi) \\
    \le & \ \dfrac{L_1 \varepsilon}{C_1}  \left[ a_1(x) f_1(\xi)+a_2(x)f_2(\xi)\right] +f(x,\xi) \\
    \le & \ \left[  \dfrac{L_1 \varepsilon}{C_1}+1\right]f(x,\xi)=:c(\varepsilon) f(x,\xi),
\end{align*}
where $c(\varepsilon) \ge 1$ and $c(\varepsilon) \to 1$, as $\varepsilon \to 0$.

In the same way, one can show that 
$$f(x,\xi)= \sum_{i=1}^N a_i(x)f_i(\xi)$$
satisfies (H5), if $a_i(x) \ge 0$ and $a_i \in \mathcal{C}^{0,1}(\Omega)$, $\forall i=1,...,N$, and there exists $j$ such that $a_j(x) \ge c_j >0$.

\item[\bf{(iii)}] Let us consider
$$f(x,\xi)=a(x)g(\xi)^{p(x)}$$
where $g \ge 1 $, $a \in \mathcal{C}^{0,1}(\Omega)$, with $a \ge C>0$, and $p \in C^0(\Omega)$.

Let $B(x,\varepsilon) \Subset \Omega$ and $x_0 \in \overline{B(x,\varepsilon)}$ such that $p(x_0) \le p(y)$, for every $y \in {B(x,\varepsilon)}$. Then,
\begin{align*}
    f(x_0, \xi) = & \ a(x_0) g(\xi)^{p(x_0)}\\
    = &\ a(x_0) g(\xi)^{p(x_0)-p(y)}g(\xi)^{p(y)}\\
    \le & \ a(x_0) g(\xi)^{p(y)},
\end{align*}
where in the last inequality we used that $g(\xi) \ge 1$ and $p(x_0)-p(y) \le 0$. Now, using the Lipschitz continuity of $a$ and the fact that $a \ge C$, we have
\begin{align*}
    f(x_0,\xi) \le & \ [a(x_0)-a(y)]g(\xi)^{p(y)} +a(y)g(\xi)^{p(y)} \\
    \le & \ L \varepsilon g(\xi)^{p(y)} +a(y)g(\xi)^{p(y)} \\
    \le & \ L \varepsilon \dfrac{a(y)}{C} g(\xi)^{p(y)} +a(y)g(\xi)^{p(y)} \\
    = & \ \left[ \dfrac{L \varepsilon}{C} +1 \right] f(y,\xi).
\end{align*}
Therefore, $f$ satisfies (H5) with 
$c(\varepsilon) := \frac{L \varepsilon}{C} +1$. For example, 
$$f(x,\xi)= a(x)(1+|\xi|^2)^\frac{p(x)}{2}.$$

\item[\bf{(iv)}] We want to show that our assumption (H5) is weaker than hypothesis (4) in \cite[Section 3]{esp-leo-pet}. 
Actually, suitably modifying the energy density considered in  \cite[Remark 3.6]{esp-leo-pet},  we construct an integral functional  for which hypothesis (4) fails but (H5) holds true. 

Let us consider $\Omega=B(0,1) \subset \mathbb{R}^2$ and the function $f: B(0,1) 
 \times \mathbb{R}^2\to \mathbb{R}$ defined by
$$f(x,\xi)=|\xi|^p+a(x)(|\xi|^q-1)+1,$$
where
\begin{center}
$a(x_1,x_2)=
    \begin{cases}
      \frac{x_2}{2}  \ \ & \text{if } x_2 >0, \\
      0 \ \ & \text{if } x_2 \le 0.
    \end{cases}
    $
\end{center}
Arguing similarly as in \cite[Remark 3.6]{esp-leo-pet}, one can prove that hypothesis (4) fails for this function.

Let $B(x,\varepsilon) \Subset \Omega$ and $\tilde{y} \in \overline{B(x,\varepsilon)}$ such that $a(\tilde{y}) \le a(y)$, for every $y \in {B(x,\varepsilon)}$. Then,
\begin{align*}
    f(\tilde{y},\xi) = & \ |\xi|^p+a(\tilde{y})(|\xi|^q-1)+1 \notag \\
    = & \ |\xi|^p+[a(\tilde{y})-a(y)](|\xi|^q-1)+1+a(y)(|\xi|^q-1) \notag \\
    = & \ f(y,\xi) + [a(\tilde{y})-a(y)](|\xi|^q-1). \label{esempio}
\end{align*}
If $|\xi| \ge 1 $, since $a(\tilde{y}) - a(y) \le 0$, we have 
$$ f(\tilde{y},\xi) \le f(y,\xi).$$
Now, assume $|\xi | < 1$. The Lipschitz continuity of $a$ implies
\begin{align*}
    f(\tilde{y},\xi) \le & \ f(y,\xi)+ |a(\tilde{y})-a(y)| ||\xi|^q-1| \\
    \le & \ f(y,\xi)+ L \varepsilon (1-|\xi|^q).
\end{align*}
Since 
$$|\xi|^p+M (|\xi|^q-1)+1 \le |\xi|^p+a(y) (|\xi|^q-1)+1,$$
for all $y \in B(x,\varepsilon)$, with $M=\max_{\overline{\Omega}} a(x)=\frac{1}{2}$, if we prove that there exists a positive constant $K$, independent both of $x$ and of $\varepsilon$, such that 
$$ 1-|\xi|^q \le K (|\xi|^p+M (|\xi|^q-1)+1) ,$$
then $f$ will satisfy (H5) with constant $c(\varepsilon)=1+L\varepsilon K$. It is sufficient to set
$$K:= \max_{|\xi| < 1} \dfrac{1-|\xi|^q}{|\xi|^p+M (|\xi|^q-1)+1} < +\infty.$$

\end{itemize}

\noindent {{\bf Acknowledgements.}
The authors are members of the Gruppo Nazionale per l’Analisi Matematica,
la Probabilità e le loro Applicazioni (GNAMPA) of the Istituto Nazionale di Alta Matematica (INdAM). A.G. Grimaldi and A. Passarelli di Napoli have been partially supported through the INdAM$-$GNAMPA 2024 Project “Interazione ottimale tra la regolarità dei coefficienti e l’anisotropia del problema in funzionali integrali a crescite non standard” (CUP: E53C23001670001). A. Passarelli di Napoli has been supported by the Università degli Studi di Napoli “Federico II”   through the project FRA-000022-ALTRI-CDA-752021-FRA-PASSARELLI and  by the Centro Nazionale per la Mobilità Sostenibile (CN00000023) - Spoke 10 Logistica Merci (CUP: E63C22000930007). In addition, A.G. Grimaldi has also been supported through the project: Geometric-Analytic Methods for PDEs and Applications (GAMPA) - funded by European Union - Next Generation EU within the PRIN 2022 program (D.D. 104 - 02/02/2022 Ministero dell'Università e della Ricerca). This manuscript reflects only the authors' views and opinions and the Ministry cannot be considered responsible for them.}

\end{document}